\documentclass[oneside,11pt]{amsart}

\usepackage{float}
\usepackage{amscd,amssymb}
\usepackage[mathscr]{eucal}
\usepackage[all]{xy}
\usepackage{mathrsfs}
\usepackage{amsfonts}
\usepackage{amsmath}
\usepackage{amsthm}
\usepackage{latexsym}
\usepackage[all]{xy}
\usepackage{phonetic}
\usepackage{todonotes}

\title{The Frobenius morphism on flag varieties, II}

\author{Alexander Samokhin}

\address{Math. Institut, Heinrich-Heine-Universit\"at, D-40204 D\"usseldorf, Germany}
\address{\it and}
\address{Institute for Information Transmission Problems, Moscow,  Russia}
\email{alexander.samokhin@gmail.com}

\newcommand{\Oo}{\mathcal O}
\newcommand{\Uu}{\mathcal U}

\newcommand{\Pp}{\mathbb P}

\newcommand{\Ff}{\mathcal F}
\newcommand{\G}{\mathcal G}
\newcommand{\Hh}{\mathcal H}
\newcommand{\Ee}{\mathcal E}

\newcommand{\Ll}{\mathcal L}

\newcommand{\D}{\mathcal D}
\newcommand{\C}{\mathcal C}
\newcommand{\T}{\mathcal T}
\newcommand{\Kk}{\mathcal K}
\newcommand{\A}{\mathcal A}
\newcommand{\B}{\mathcal B}

\newcommand*{\RHom}{\mathop{\mathrm RHom}\nolimits}
\newcommand*{\Hom}{\mathop{\mathrm Hom}\nolimits}
\newcommand*{\Dd}{\mathop{\mathrm D\kern0pt}\nolimits}
\newcommand*{\DD}{\mathop{\mathbb D\kern0pt}\nolimits}

\newcommand*{\Ext}{\mathop{\mathrm Ext}\nolimits}

\newtheorem{theorem}{Theorem}[section]
\newtheorem{corollary}{Corollary}[section]
\newtheorem{lemma}{Lemma}[section]
\newtheorem{proposition}{Proposition}[section]
\newtheorem{claim}{Claim}[section]
\newtheorem{conjecture}{Conjecture}[section]
\newtheorem{remark}{Remark}[section]
\newtheorem{definition}{Definition}[section]	

\numberwithin{equation}{section}

\long\def\comment#1{}

\begin{document}
	
\maketitle

\begin{abstract}
In this paper, which is the sequel to \cite{FrobflagvarI}, we study the Frobenius pushforward of the structure sheaf on the adjoint varieties in type ${\bf A}_3$ and ${\bf A}_4$. We show that this pushforward sheaf decomposes into a direct sum of indecomposable bundles and explicitly determine this set that does not depend of the characteristic. In accordance with the results of \cite{Vanth}, this set forms a strong full exceptional collection in the derived category of coherent sheaves. These computations lead to a natural conjectural answer in the general case that we state at the end.
\end{abstract}

\vspace*{0.5cm}

\section{Introduction}

\vspace*{0.5cm}
Let ${\sf V}$ be a vector space of dimension $n$ over an algebraically closed field $\sf k$ of characteristic $p$, and let $X_n$ denote the partial flag variety ${\rm F}_{1,n-1,{\sf V}}$ of type $(1,n-1)$. Recall that given a variety $X$ over $\sf k$, the sheaf of small differential operators $\D _X^{(1)}$ on $X$ is a coherent sheaf of algebras isomorphic to ${\mathcal End}_{\Oo _X}({\sf F}_{\ast}\Oo _X)$, where ${\sf F}$ is the Frobenius morphism.
It was shown in \cite{Vanth} that ${\rm H}^i(X_n,{\mathcal D}^{(1)}_{X_n})=0$ for $i>0$ in all characteristics. Coupled with the results of \cite{BMR}, this vanishing theorem gives that for $p>n$ the bundle ${\sf F}_{\ast}\Oo _{X_n}$ is {\it tilting} on ${X_n}$.

The proof in \cite{Vanth} was rather implicit, however, as it deduced the higher cohomology vanishing of the sheaf ${\mathcal D}^{(1)}_{X_n}$ from
the general properties of sheaves of crystalline differential operators that 
turn out to be particularly nice on the varieties ${X_n}$. The argument presented in {\it loc.cit.} also covers the case of smooth quadrics for odd primes; on the other hand, in the latter case there is an explicit description of the decomposition of the Frobenius pushforward of a line bundle in arbitrary characteristic (see \cite{Ach} and \cite{Lan}). The tilting property for appropriate primes follows as well from those decompositions. 

The goal of the present paper is to explicitly compute the summands appearing in ${\sf F}_{\ast}\Oo _{X_n}$. We use the approach developed in the previous paper \cite{FrobflagvarI}. To make the present paper self--contained, we include the preliminary material (without proofs) from {\it loc.cit.} that occupies Sections \ref{sec:Cohomology_on_G/B} -- \ref{sec:res_of_Delta}. The reader familiar with these notions may skip directly to Section \ref{sec:adjoint_var_rank3}. In particular, we provide a detailed argument allowing to compute the indecomposable summands of ${\sf F}_{\ast}\Oo _{X_n}$ in the case of small ranks $n=4$ (Theorem \ref{th:F-theorem_X_4}) and $n=5$ (Theorem \ref{th:F-theorem_X_5}). In consistency with \cite{Vanth}, these decompositions show that ${\sf F}_{\ast}\Oo _{X_n}$ is a tilting bundle on $X_n$ for $n=4,5$ and $p>n$. The approach to construct the decomposition of ${\sf F}_{\ast}\Oo _{X_n}$ in these low rank cases is, in fact, uniform; it was to a large extent inspired by the seminal paper \cite{Kap}. In the final Section \ref{sec:adjoint_var} we give a conjectural description of the decomposition of ${\sf F}_{\ast}\Oo _{X_n}$ in the general case. 

\subsection*{Notation}
Throughout we fix a perfect field $k$ of characteristic $p>0$.
Given a split semisimple simply connected algebraic group $\bf G$ over $k$, let $\bf T$ denote a maximal torus of $\bf G$, and let ${\bf T}\subset {\bf B}$ be a Borel subgroup containing $\bf T$. The flag variety of Borel subgroups in $\bf G$ is denoted ${\bf G/B}$.
Denote ${\rm X}({\bf T})$ the weight lattice, and let $\rm R$ and  $\rm R ^{\vee}$ denote the root and coroot lattices, respectively. Let $\rm S$ be the set of simple roots relative to the choice of a Borel subgroup than contains $\bf T$. 
The Weyl group ${\mathcal W}={\rm N}({\bf T})/{\bf T}$ acts on $X({\bf T})$ via the dot--action: if $w\in {\mathcal W}$, and 
$\lambda \in {\rm X}({\bf T})$, then $w\cdot \lambda = w(\lambda + \rho) - \rho$. A parabolic subgroup of $\bf G$ is denoted by $\bf P$. For a simple root $\alpha \in {\rm S}$, denote ${\bf P}_{\alpha}\subset \bf G$ the corresponding minimal parabolic subgroup.  Given a weight $\lambda \in  {\rm X}({\bf T})$, denote $\Ll _{\lambda}$ the corresponding line bundle on ${\bf G}/{\bf B}$.  The half sum of the positive roots (the sum of fundamental weights) is denoted by $\rho$. Given a dominant 
weight $\lambda \in  {\rm X}({\bf T})$, the induced module ${\rm Ind}_{\bf B}^{\bf G}\lambda$ is denoted $\nabla _{\lambda}$, the Weyl module, which is dual to induced module, is denoted $\Delta _{\lambda}$, and 
the simple module with the highest weight $\lambda$ is denoted ${\sf L}_{\lambda}$. Given a variety $X$ and $n\in \mathbb N$, denote ${\sf F}_n$ the $n$--th iteration of the absolute  Frobenius morphism ${\sf F}_n: X\rightarrow X$. For a vector space ${\sf V}$ over $k$ its $n$-th Frobenius twist ${\sf F}_n^{\ast}{\sf V}$ is denoted ${\sf V}^{[n]}$. All the functors are supposed to be derived, i.e., given a morphism $f:X\rightarrow Y$ between two schemes, we write $f_{\ast},f^{\ast}$ for the corresponding derived functors of push--forwards and pull--backs.

\subsection*{\bf Acknowledgements}
The first drafts of this paper have been written in 2012--13 while the author benefited from generous supports of the SFB/Transregio 45 at the University of Mainz, and of the SFB 878 at the University of M\"unster. It is a great pleasure to thank Manuel Blickle and Christopher Deninger for their invitations. We would also like to thank Michel Van den Bergh for his interest in this work and for the inspiring paper \cite{VdBFrob} that gave the author an incentive to reanimate this project, and the organizers of the
conference "Triangulated categories and geometry -- in honour of A.Neeman" in Bielefeld in May, 2017 at which this paper took its present shape. The author gratefully acknowledges support from the strategic research fund of the Heinrich-Heine-Universit\"at D\"usseldorf (grant SFF F-2015/946-8).

\vspace*{0.5cm}

\section{Cohomology of line bundles on ${\bf G}/{\bf B}$}\label{sec:Cohomology_on_G/B}

\vspace*{0.5cm}

\subsection{\bf Flag varieties of Chevalley groups over ${\mathbb Z}$}

Let ${\mathbb G}\rightarrow \mathbb Z$ be a semisimple Chevalley group scheme (a smooth affine group scheme over ${\rm Spec}(\mathbb Z)$ whose geometric fibres are connected semisimple algebraic groups), and ${\mathbb G}/{\mathbb B}\rightarrow \mathbb Z$ be the corresponding Chevalley flag scheme (resp., ${\mathbb P}\subset \mathbb G$ the corresponding parabolic subgroup scheme over ${\mathbb Z}$). 
Then ${\mathbb G/\mathbb P}\rightarrow {\rm Spec}({\mathbb Z})$ is flat and the line bundle $\Ll$ on ${\bf G/P}$ also comes from a line bundle $\mathbb L$ on ${\mathbb G/\mathbb P}$. Let $k$ be a field of arbitrary characteristic, and ${\bf G/B}\rightarrow {\rm Spec}(k)$ be the flag variety obtained by base change along ${\rm Spec}(k)\rightarrow {\rm Spec}(\mathbb Z)$.

\subsection{\bf Bott's vanishing theorem}\label{subsec:cohlinbunflags}

We recall first the classical Bott's theorem (see \cite{Dem}). Let ${\mathbb G}\rightarrow \mathbb Z$ be a semisimple Chevalley group scheme as above. Assume given a weight $\chi \in X({\mathbb T})$, and let $\Ll _{\chi}$ be the corresponding line bundle on ${\mathbb G}/{\mathbb B}$. The weight $\chi$ is called {\it singular}, if it lies on a wall of some Weyl chamber defined by $\langle -, \alpha ^{\vee}\rangle =0$ for some coroot $\alpha ^{\vee}\in {\rm R}^{\vee}$. Weights, which are not singular, are called {\it regular}. Let $k$ be a field of characteristic zero, and ${\bf G}/{\bf B}\rightarrow {\rm Spec}(k)$ the corresponding flag variety over $k$. The weight $\chi \in X({\bf T})$ defines a line bundle $\Ll _{\chi}$ on 
${\bf G}/{\bf B}$.

\begin{theorem}\cite[Theorem 2]{Dem}\label{th:Bott-Demazure_th}

\vspace*{0.2cm}

\begin{itemize}

\vspace*{0.2cm}

\item[(a)] If $\chi +\rho$ is singular, then ${\rm H}^i({\bf G}/{\bf B},\Ll _{\chi})= 0$ for all $i$.

\vspace*{0.2cm}

\item[(b)] If If $\chi + \rho$  is regular and dominant, then ${\rm H}^i({\bf G}/{\bf B},\Ll _{\chi}) = 0$ for $i>0$.

\vspace*{0.2cm}

\item[(c)]  If $\chi + \rho$  is regular, then ${\rm H}^i({\bf G}/{\bf B},\Ll _{\chi})\neq 0$ for the unique degree $i$, which is equal to $l(w)$. Here $l(w)$ is the length of an element of the Weyl group that takes $\chi$ to the dominant chamber, i.e. $w\cdot \chi \in X_{+}({\bf T})$. The cohomology group ${\rm H}^{l(w)}({\bf G}/{\bf B},\Ll _{\chi})$ is the irreducible $\bf G$--module of highest weight $w\cdot \chi$.

\end{itemize}

\end{theorem}

\subsection{Cohomology of line bundles} 

Some bits of Theorem \ref{th:Bott-Demazure_th} are still true over $\mathbb Z$: if a weight
$\chi$ is such that $\langle \chi + \rho, \alpha ^{\vee}\rangle =0$ for some simple root $\alpha$, then the corresponding line bundle is acyclic. Indeed, Lemma from  \cite[Section 2]{Dem} holds over fields of arbitrary characteristic. 

\begin{theorem}[Kempf's vanishing theorem]\label{th:Kempf}
Let $\chi \in X(\bf T)$, i.e. $\langle \chi ,\alpha ^{\vee}\rangle \geq 0$ for all simple coroots $\alpha ^{\vee}$. Then ${\rm H}^i({\bf G}/{\bf B},\Ll _{\chi})=0$ for $i>0$.
\end{theorem}

Besides this, however, very little of Theorem \ref{th:Bott-Demazure_th} holds over $\mathbb Z$ \cite[Part II, Chapter 5]{Jan}. However, it still holds for weights lying in the interior of the bottom alcove in the  dominant chamber \cite[Theorem 2.3 and Corollary 2.4]{And}:

\begin{theorem}\label{th:AndthII}
If $\chi$ is a weight such that for a simple root $\alpha$ one has
$0\leq \langle \chi + \rho,\alpha ^{\vee}\rangle \leq p$ then 

\vspace*{0.2cm}

\begin{equation}
{\rm H}^i({\bf G}/{\bf B},\Ll _{\chi}) = {\rm H}^{i+1}({\bf G}/{\bf
  B},\Ll _{s_{\alpha}\cdot \chi}).
\end{equation}

\vspace*{0.2cm}

\end{theorem}

The following theorem is used throughout in all the calculations concerning cohomology of line bundles (see \cite[Corollary 3.2]{An}):

\begin{theorem}\label{th:Andcor}
Let $\chi$ be a weight. If either $\langle \chi ,\alpha ^{\vee}\rangle
\geq -p$ or  $\langle \chi ,\alpha ^{\vee}\rangle = -ap^n-1$ for some
$a, n\in {\bf N}$ and $a<p$ then 

\vspace*{0.2cm}

\begin{equation}
{\rm H}^i({\bf G}/{\bf B},\Ll _{\chi}) = {\rm H}^{i-1}({\bf G}/{\bf
  B},\Ll _{s_{\alpha}\cdot \chi}).
\end{equation}

\vspace*{0.2cm}

\end{theorem}

\vspace*{0.5cm}

\section{Semiorthogonal decompositions, mutations, and exceptional collections}\label{sec:SOD_mutations}

\vspace*{0.5cm}

\subsection{Semiorthogonal decompositions}

Let $k$ be a field. Assume given a $k$--linear triangulated category $\D$, equipped with a shift functor $[1]\colon \D \rightarrow \D$.  For two
objects $A, B \in \D$ let $\Hom ^{\bullet}_{\D}(A,B)$ be
the graded $k$-vector space $\oplus _{i\in \mathbb Z}\Hom _{\D}(A,B[i])$. 
Let $\A \subset \D$ be a full triangulated subcategory,
that is a full subcategory of $\D$ which is closed under shifts.

The original source for most of the definitions and statements in this section is \cite{Bo}. 
We follow the expositions of  \cite[Section 2.1]{Efim} and \cite[Section 2.2]{Kuzrat}.

\begin{definition}\label{def:orthogonalcat}
The right orthogonal $\A ^{\perp}\subset \D$ is defined to be
the full subcategory

\vspace*{0.2cm}

\begin{equation}
\A^{\perp} = \{B \in \D \colon \Hom _{\D}(A,B) = 0 \}
\end{equation}

\vspace*{0.2cm}

\noindent for all $A \in \A$. The left orthogonal $^{\perp}\A$ is defined similarly. 

\end{definition}

\begin{definition}\label{def:admissible}
A full triangulated subcategory $\A$ of $\D$ is called
{\it right admissible} if the inclusion functor $\A\hookrightarrow \D$ has a right adjoint. Similarly, $\A$ is called {\it left
  admissible} if the inclusion functor has a left adjoint. Finally,
$\A$ is {\it admissible} if it is both right and
left admissible.
\end{definition}

If a full triangulated category $\A\subset \D$ is right admissible then every object $X\in \D$ fits into a distinguished triangle
  
\vspace*{0.2cm}
  
\begin{equation}
\dots \longrightarrow  Y\longrightarrow X\longrightarrow Z\longrightarrow Y[1]\rightarrow \dots
\end{equation}

\vspace*{0.2cm}

\noindent with $Y\in \A$ and $Z\in \A^{\perp}$. One then
says that there is a semiorthogonal decomposition of $\D$ into
the subcategories $(\A^{\perp}, \ \A)$. More generally,
assume given a sequence of full triangulated subcategories $\A _1,\dots,\A_n \subset \D$. Denote $\langle \A _1,\dots,\A_n\rangle$ the triangulated subcategory of $\D$ generated by $\A_1,\dots,\A_n$.

\begin{definition}\label{def:semdecomposition}
A sequence $(\A_1,\dots,\A_n)$ of admissible subcategories of
$\D$ is called {\it semiorthogonal} if $\A _i\subset \A_j^{\perp}$ for $1\leq i < j\leq n$,
and $\A_i\subset {^{\perp}\A_j}$ for $1\leq j < i\leq n$.
The sequence $(\A_1,\dots,\A_n)$ is called a {\it semiorthogonal
  decomposition} of $\D$ if $\langle \A_1, \dots, \A_n
\rangle^{\perp} = 0$, that is $\D = \langle \A_1,\dots,\A_n\rangle$.
\end{definition}

\vspace{0.2cm}

The above definition is equivalent to:

\begin{definition}\label{def:semdecomposition-filtration}
A semiorthogonal decomposition of a triangulated category $\D$ is a sequence of full triangulated subcategories $(\A_1,\dots,\A_n)$ in $\D$ such that $\A _i\subset \A_j^{\perp}$ for $1\leq i < j\leq n$ and for every object $X\in \D$ there exists a chain of morphisms in $\D$,

\vspace*{0.2cm}

\[
\xymatrix@C=.5em{
0_{\ } \ar@{=}[r] & X_n \ar[rrrr] &&&& X_{n-1} \ar[rrrr] \ar[dll] &&&& X_{n-2}
\ar[rr] \ar[dll] && \ldots \ar[rr] && X_{1}
\ar[rrrr] &&&& X_0 \ar[dll] \ar@{=}[r] &  X_{\ } \\
&&& A_{n-1} \ar@{->}^{[1]}[ull] &&&& A_{n-2} \ar@{->}^{[1]}[ull] &&&&&&&& A_0\ar@{->}^{[1]}[ull] 
}
\]

\vspace*{0.2cm}

such that a cone $A_k$ of the morphism $X_k\rightarrow X_{k-1}$ belongs to $\A _k$ for $k=1,\dots ,n$.

\vspace*{0.2cm}

\end{definition}


\subsection{Mutations}

Let $\D$ be a triangulated category and assume $\D$ admits a semiorthogonal decomposition 
$\D = \langle \A, \B \rangle$.

\begin{definition}\label{def:left-right_mutation}
The left mutation of $\B$ through $\A$ is defined to be ${\bf L}_{\A}(\B): = \A ^{\perp}$. The 
right mutation of $\A$ through $\B$ is defined to be ${\bf R}_{\B}(\A): = {^{\perp}\B}$.
\end{definition}

One obtains semiorthogonal decompositions $\D = \langle {\bf L}_{\A}(\B), \A\rangle$ and 
$\D = \langle \A , {\bf R}_{\B}(\A)\rangle$.\par 

Let $\A$ be an admissible subcategory of $\D$, and $i : \A\rightarrow \D$ the embedding functor.
It admits a left and a right adjoint functors $\D\rightarrow \A$, the subcategory $\A$ being admissible; denote them $i^{\ast}$ and $i^{!}$, respectively.  Given an object $F \in \D$, define the left mutation ${\bf L}_{\A}(F)$ and the right mutations ${\bf R}_{\A}(F)$  of $F$ through $\A$ by 

\vspace*{0.2cm}

\begin{equation}\label{eq:left-right_mutations_of_an_object}
{\bf L}_{\A}(F): = {\sf Cone}(ii^{!}(F)\rightarrow F), \qquad \qquad {\bf R}_{\A}(F): = {\sf Cone}(F\rightarrow ii^{\ast}(F))[-1].
\end{equation}

\vspace*{0.2cm}

One the proves:

\begin{lemma}\cite[Lemma 2.7]{Kuzrat}\label{lem:mutations_made_explicit}
There are equivalences ${\bf L}_{\A} :\B\simeq \D/\A\simeq {\bf L}_{\A}(\B)$ and ${\bf R}_{\A} :\A\simeq \D/\B\simeq {\bf R}_{\A}(\B)$.
\end{lemma}

\begin{proposition}\cite[Proposition 2.3]{Bo}\label{prop:right-left_mutations_inverse}
Let $\D = \langle \A, \B \rangle$ be as above. Right and left mutations are mutually inverse to each other, i.e. ${\bf R}_{\A}{\bf L}_{\A}\simeq {\rm id}_{\A}$, and ${\bf L}_{\B}{\bf R}_{\B}\simeq {\rm id}_{\B}$.
\end{proposition}

\begin{definition}\label{def:left_and_right_dual_decompositions}
Let $\D = \langle \A _1,\dots,\A_n\rangle$ be a semiorthogonal decomposition. The left dual semiorthogonal decomposition  $\D = \langle \B _n,\dots,\B_1\rangle$ is defined by 

\vspace*{0.2cm}

\begin{equation}
\B _i : = {\bf L}_{\A _1}{\bf L}_{\A _2}\dots {\bf L}_{\A _{i-1}}\A _i = {\bf L}_{\langle \A _1,\dots,\A_{i-1}\rangle}\A _i , \quad 1\leq i \leq n .
\end{equation}

\vspace*{0.2cm}

The right dual semiorthogonal decomposition $\D = \langle \C _n,\dots,\C_1\rangle$ is defined by 

\vspace*{0.2cm}

\begin{equation}
\C _i : = {\bf R}_{\A _n}{\bf R}_{\A _{n-1}}\dots {\bf R}_{\A _{i+1}}\A _i = {\bf R}_{\langle \A _{i+1},\dots,\A_{n}\rangle}\A _i , \quad 1\leq i \leq n .
\end{equation}

\vspace*{0.2cm}

\end{definition}

\begin{lemma}\cite[Lemma 2.10]{Kuzrat}\label{lem:mutations_of_completely_orthogonal_subcategories}
Let $\D = \langle \A _1,\dots,\A_n\rangle$ be a semiorthogonal decomposition such that the components $\A _k$ and $\A _{k+1}$ are completely orthogonal, i.e., $\Hom _{\D}(\A _k,\A _{k+1}) = 0$ and $\Hom _{\D}(\A _{k+1},\A _k) = 0$. Then 

\vspace*{0.2cm}

\begin{equation}
{\bf L}_{\A_k}\A _{k+1} = \A _{k+1} \qquad \qquad and  \qquad \qquad {\bf R}_{\A_{k+1}}\A _k = \A _k,
\end{equation}

\vspace*{0.2cm}

and both the left mutation of $\A _{k+1}$ through ${\A_k}$ and the right mutation of $\A_k$ through $\A _{k+1}$ boil down to a permutation and

\vspace*{0.2cm}

\begin{equation}
\D = \langle \A _1,\dots,\A _{k-1},\A _{k+1},\A _k, \A _{k+2},\dots \A_n\rangle
\end{equation}

\vspace*{0.2cm}

is the resulting semiorthogonal decomposition of $\D$.

\end{lemma}

\subsection{Exceptional collections}

Exceptional collections in $\sf k$--linear triangulated categories are a special case of semiorthogonal decompositions with each component of the decomposition being equivalent to $\Dd ^b({\sf Vect}-{\sf k})$. The above properties of mutations thus specialize to this special case.
Still, there are new features appearing as shown in Subsection \ref{subsec:blockcol}.

\begin{definition}\label{def:exceptcollection}
An object $E \in \D$ of a $\sf k$--linear triangulated category $\D$ is said to be exceptional if there is an isomorphism of graded $\sf k$-algebras 

\vspace{0.2cm}

\begin{equation}
\Hom _{\D}^{\bullet}(E,E) = {\sf k}.
\end{equation}

\vspace{0.2cm}

A collection of exceptional objects $(E_0,\dots,E_n)$ in $\D$ is called 
exceptional if for $1 \leq i < j \leq n$ one has

\vspace{0.2cm}

\begin{equation}
\Hom _{\D}^{\bullet}(E_j,E_i) = 0.
\end{equation}

\vspace{0.2cm}

\end{definition}
Denote $\langle E_0,\dots,E_n \rangle \subset {\D}$ the full
triangulated subcategory generated by the objects $E_0,\dots,E_n$. One 
proves \cite[Theorem 3.2]{Bo} that such a category is admissible. 
The collection $(E_0,\dots,E_n)$ in $\D$ is said to be {\it full} if 
$\langle E_0,\dots,E_n \rangle ^{\perp} = 0$, in other words ${\D}
= \langle E_0,\dots,E_n \rangle$.

 If $\A \subset \D$ is generated by an exceptional object $E$, then by 
 (\ref{eq:left-right_mutations_of_an_object}) the left and right mutations of an object $F\in \D$ through $\A$ are given by the following distinguished triangles:

\vspace*{0.2cm}

\begin{equation}
\RHom _{\D}(E,F)\otimes E\rightarrow F\rightarrow {\bf L}_{\langle E\rangle}(F), \qquad 
{\bf R}_{\langle E\rangle}(F)\rightarrow F\rightarrow \RHom _{\D}(F,E)^{\ast}\otimes E .
\end{equation}

\vspace*{0.2cm}


\subsection{Block collections and block mutations}\label{subsec:blockcol}


\vspace*{0.2cm}

The results of this section are needed for the subsequent Theorem \ref{th:Frobdecomposclaim}.
We follow the exposition of \cite[Section 4]{BrS}.

\begin{definition}\label{def:d-block_collection}
A $d$--block exceptional collection is an exceptional collection $\mathbb E = (E_1,\dots, E_n)$ together with a partition of $\mathbb E$  into $d$ subcollections

\vspace*{0.2cm}

\begin{equation}
\mathbb E = (\mathbb E _1,\dots ,\mathbb E _d),
\end{equation}

\vspace*{0.2cm}

called blocks, such that the objects in each block $\mathbb E _i$ are mutually orthogonal, i.e.
$\Hom ^{\bullet}(E,E') =  0 = \Hom ^{\bullet}(E',E)$ for any $E, E'\in \mathbb E _i$.
\end{definition}

For each integer $1< i\leq d$ we can define an operation $\tau _i$ on $d$--block collections in $\D$ by the rule

\vspace*{0.2cm}

\begin{eqnarray}\label{eq:mutations_block_collections}
& \tau _i (\mathbb E _1,\dots ,\dots \mathbb E _{i-2},\mathbb E _{i-1},\mathbb E _i,\mathbb E _{i+1},\dots \mathbb E _d) = \\
& (\mathbb E _1,\dots ,\dots \mathbb E _{i-2},{\bf L}_{\mathbb E _{i-1}},(\mathbb E _i)[-1],
\mathbb E _{i-1},\mathbb E _{i+1},\dots \mathbb E _d). \nonumber
\end{eqnarray}

\vspace*{0.2cm}

Here, if $\mathbb E _i = (E_{a+1},\dots ,E_b)$ then by definition 

\vspace*{0.2cm}

\begin{equation}
{\bf L}_{\mathbb E _{i-1}}(\mathbb E _i)= ({\bf L}_{\mathbb E _{i-1}}E_{a+1},\dots, {\bf L}_{\mathbb E _{i-1}}E_{b}).
\end{equation}

\vspace*{0.2cm}

\begin{remark}\label{rem:shift_in_mutation}
{\rm Note the shift by $[-1]$ in (\ref{eq:mutations_block_collections}) at ${\bf L}_{\mathbb E _{i-1}},(\mathbb E _i)$; this  will ensure that in the situations below the block mutations $\tau _i$ preserve collections of pure objects.}
\end{remark}

\vspace*{0.2cm}

Recall that a Serre functor (see \cite{BK}) on a $\bf k$--linear triangulated category $\D$ is an autoequivalence ${\mathbb S}_{\D}$ of $\D$ for which there are natural isomorphisms $\Hom _{\D}(E,F) = \Hom _{\D}(F,{\mathbb S}_{\D}(E))^{\ast}$ for $E,F\in \D$. If a Serre functor exists then it is unique up to isomorphism \cite[Proposition 3.4]{BK}.

Given a smooth algebraic variety $X$ over a field $k$, denote $\Dd ^b(X)$ the bounded derived category of coherent sheaves. It is a $k$--linear triangulated category. Let $\omega _X$ be the canonical line bundle on $X$. If $\D = \Dd ^b(X)$ for a smooth projective variety $X$ of dimension $d$, then ${\mathbb S}_{\D} = (-\otimes \omega _X)[d]$.

\begin{theorem}\cite[Theorem 4.5]{BrS}\label{th:mutations_of_block_collections}
Suppose $\mathbb E = (\mathbb E _1,\dots ,\mathbb E _d)$ is a full $d$--block collection and take $1<i\leq d$. Suppose $\D$ is is equipped with a $t$--structure that is preserved by the autoequivalence ${\mathbb S}_{\D}[1-d]$. Then 

\vspace*{0.2cm}

\begin{itemize}

\vspace*{0.2cm}

\item $\mathbb E$ pure implies $\tau _i(\mathbb E)$ pure.

\vspace*{0.2cm}

\item $\mathbb E$ pure implies $\tau _i(\mathbb E)$ strong.

\end{itemize}

\end{theorem}

For our needs, Theorem \ref{th:mutations_of_block_collections} means the following. Let $X$ be a smooth variety of dimension $d-1$, and $\D = \Dd ^b(X)$ equipped with the standard $t$--structure 
$(\Dd ^b(X)^{\leq 0},\Dd ^b(X)^{\geq 0})$. Assume there exists a $d$--block full exceptional collection in $\Dd ^b(X)$ consisting of pure objects, that is, of coherent sheaves in this case.
Then the autoequivalence ${\mathbb S}_{\D}[1-d]$ is just tensoring with $\omega _X$, thus the condition of Theorem \ref{th:mutations_of_block_collections} is immediately satisfied. In this setting, Theorem \ref{th:mutations_of_block_collections} then means that left and right mutations are, too, exceptional collections consisting of coherent sheaves.

\begin{lemma}\cite[Lemma 2.11]{Kuzrat}\label{lem:mutations_canonical_class}
Assume given a semiorthogonal decomposition $\D = \langle \A ,\B\rangle $. Then 

\vspace*{0.2cm}

\begin{equation}
{\bf L}_{\A}(\B)=\B \otimes \omega _X \qquad \qquad and \qquad \qquad  {\bf R}_{\A}(\B)= \A \otimes \omega _X^{-1}.
\end{equation}

\vspace*{0.2cm}

\end{lemma} 


Let $\Ee$ be a vector bundle of rank $r$ on $X$, and consider the associated projective bundle $\pi : \Pp (\Ee)\rightarrow X$. Denote $\Oo _{\pi}(-1)$ the invertible line bundle on $\Pp (\Ee)$ of relative degree $-1$, such that $\pi _{\ast}\Oo _{\pi}(1)=\Ee ^{\ast}$.  One has, \cite{Or}:

\vspace{0.2cm}

\begin{theorem}\label{th:Orvlovth}
The category $\Dd ^b(\Pp (\Ee))$ has a semiorthogonal decomposition: 


\begin{equation}
\Dd ^b(\Pp (\Ee)) = \langle \pi ^{\ast}\Dd ^b(X)\otimes \Oo _{\pi}(-r+1),\dots ,  \pi ^{\ast}\Dd ^b(X)\otimes \Oo _{\pi}(-1),\pi ^{\ast}\Dd ^b(X)\rangle .
\end{equation}

\vspace{0.1cm}

\end{theorem}

\subsection{Dual exceptional collections}


\vspace*{0.2cm}

\begin{definition}\label{def:left-right_dual_exceptional_collections}
Let $X$ be a smooth variety, and assume given an exceptional collection $(E_0,\dots,E_n)$ in $\Dd ^{b}(X)$. The right dual exceptional collection $(F_n,\dots, F_0)$ to $(E_0,\dots,E_n)$ is defined as 

\vspace*{0.2cm}

\begin{equation}
F_i: = {\bf R}_{\langle E_{i+1},\dots,E_{n} \rangle}E_i, \quad {\rm for} \quad 1\leq i\leq n.
\end{equation}

\vspace*{0.2cm}

The left dual exceptional collection $(G_n,\dots, G_0)$ to $(E_0,\dots,E_n)$ is defined as 

\vspace*{0.2cm}

\begin{equation}
G_i: = {\bf L}_{\langle E_{1},\dots,E_{i-1} \rangle}E_i, \quad {\rm for} \quad 1\leq i\leq n.
\end{equation}

\vspace*{0.2cm}

\end{definition}

\begin{proposition}\cite[Proposition 2.15]{Efim}\label{prop:dual_exc_coll_characterization}
Let $(E_0,\dots,E_n)$ be a semiorthogonal decomposition in a triangulated category $\D$. The left dual exceptional collection $\langle F_n, \dots, F_0\rangle$ is uniquely determined by the following property:

\vspace*{0.2cm}

\begin{equation}
\Hom ^l_{\D}(E_i,F _j) = \left\{
   \begin{array}{l}
    {\sf k}, \quad {\rm for} \quad l=0, \  i=j,\\
    0, \quad {\rm otherwise}. \\
   \end{array}
  \right.
\end{equation}

\vspace*{0.2cm}

Similarly, the right dual exceptional collection $\langle G _n,\dots, G _0\rangle $ is uniquely determined by the following property:

\vspace*{0.2cm}

\begin{equation}
\Hom ^l_{\D}(G_i,E_j) = \left\{
   \begin{array}{l}
    {\sf k}, \quad {\rm for} \quad l=0, \  i=j,\\
    0, \quad {\rm otherwise}. \\
   \end{array}
  \right.
\end{equation}

\vspace*{0.2cm}

\end{proposition}

\vspace*{0.5cm}


\section{Resolutions of the diagonal and the decomposition of ${\sf F}_{\ast}\Oo _X$}\label{sec:res_of_Delta}


\vspace*{0.5cm}
Let $X$ be a smooth variety. Given two (admissible) subcategories $\A$ and $\B$ of $\Dd ^b(X)$, define $\A \boxtimes \B\subset \Dd ^b(X)$ to be the minimal triangulated subcategory of $\Dd ^b(X)$ that contains all the objects $(A\boxtimes B| A\in \A, B\in \B)$. The results of \cite{Kuzbasechange} on base change for semiorthogonal decompositions imply:

\begin{theorem}\label{th:Kuzbasechange}
Let $X$ be as above, and assume given a semiorthogonal decomposition $\langle \A _1,\dots, \A _m\rangle$ of $\Dd ^b(X)$. Let $\langle \C _m,\dots ,\C_1\rangle$ be the right dual semiorthogonal decomposition of $\Dd ^b(X)$ as in Definition \ref{def:left_and_right_dual_decompositions}. 
Then the structure sheaf of the diagonal  $\Delta _{\ast}\Oo _X$ admits (cf. Definition \ref{def:semdecomposition-filtration}) a decomposition $0=D_m\rightarrow D_{m-1}\rightarrow \dots \rightarrow D_1\rightarrow D_0=\Delta _{\ast}\Oo _X$ in $\Dd ^b(X\times X)$, such that ${\rm Cone}(D_i\rightarrow D_{i-1})\in \A _i\boxtimes \C _i^{\vee}\subset \Dd ^b(X\times X)$.
\end{theorem}


\comment{
\vspace*{0.2cm}

\begin{eqnarray}
& \Ext ^{\ast}(\G \boxtimes \A _j^{\vee},P_i) = \Ext ^{\ast}(p_1^{\ast}\G \otimes p_2^{\ast}\A _j^{\vee},P_i) =  \Ext ^{\ast}(p_1^{\ast}\G ,P_i\otimes p_2^{\ast}\ A _j) = \\
& \Ext ^{\ast} (\G ,{p_1}_{\ast}(P_i\otimes p_2^{\ast}\A _j)) = : \Phi _{P_i}(\A _j). \nonumber
\end{eqnarray}

\vspace*{0.2cm}
}


\begin{corollary}\label{cor:resolution_of_Delta}
Let $X$ be a smooth projective variety, and $(\Ee _0,\dots ,\Ee _m)$ be a full exceptional collection in $\Dd ^b(X)$ with $(\Ff _m, \dots, \Ff _1)$ being its right dual. Then the structure sheaf of the diagonal  $\Delta _{\ast}\Oo _X$ admits  a decomposition $0=D_{m+1}\rightarrow D_m\rightarrow \dots \rightarrow D_1\rightarrow D_0=\Delta _{\ast}\Oo _X$ in $\Dd ^b(X\times X)$, such that a cone of each morphism $D _i\rightarrow D_{i-1}$ is quasiisomorphic to $\Ee _i\boxtimes \Ff  _i^{\vee}$, where $\Ff  _i^{\vee}={\mathcal RHom}(\Ff  _i,\Oo _X)$ is the dual object.

In particular, for any object $\G $ of $\Dd ^b(X)$ there is a spectral sequence 

\vspace*{0.2cm}

\begin{equation}\label{eq:ExcCollSpecSeq}
{\rm E}_1^{p,q}: = {\mathbb H}^{p+q}(X,\G\otimes  \Ff _p^{\vee})\otimes  \Ee _p \Rightarrow \G .
\end{equation}

\vspace*{0.2cm}
\end{corollary}

Assembling together all the previous definitions and statements, we immediately obtain:

\begin{theorem}\label{th:Frobdecomposclaim}
Let $X$ be a smooth variety of dimension $d-1$ over a $\sf k$ of characteristic $p$. Fix an $m\geq 1$ and consider the $m$--th Frobenius morphism ${\sf F}_m$. Assume given a $d$--block (cf. Definition \ref{def:d-block_collection}) full exceptional collection ${\mathbb E} = (\mathbb E _{-d+1},\dots ,\mathbb E _0)$ in $\Dd ^b(X)$ consisting of  coherent sheaves.  Furthermore, assume that for any exceptional object $\Ee \in \mathbb E _i$ and $-d+1\leq i\leq 0$, one has 
${\rm H}^j(X,{\sf F}_m^{\ast}\Ee )=0$ for $j\neq -i$. Denote ${\mathbb G}=(\mathbb G _0,\dots ,\mathbb G _{-d+1})$ the right dual collection. Then 

\vspace*{0.2cm}

\begin{itemize}

\vspace*{0.2cm}

\item[(1)] The right dual collection ${\mathbb G}=(\mathbb G _0,\dots ,\mathbb G _{-d+1})$ is a  $d$--block full exceptional collection.

\vspace*{0.2cm}

\item[(2)] For an exceptional vector bundle $\G \in {\mathbb G}_i$ the corresponding shift is equal to $-i$.

\vspace*{0.2cm}

\item [(3)]There is a decomposition of the bundle ${{\sf F}_n}_{\ast}\Oo _X$ into the direct sum:

\vspace*{0.2cm}

\begin{equation}\label{eq:Decomposition_of_F_*Oo_X}
{{\sf F}_n}_{\ast}\Oo _X = \bigoplus _{i=1}^{i=d} \bigoplus _{\Ee \in \mathbb E _i,  \G \in \mathbb G_{i-d}}{\rm H}^i(X,{\sf F}_n^{\ast}\Ee)\otimes \G^{\vee},
\end{equation}

\vspace*{0.2cm}

and in the inner sum of (\ref{eq:Decomposition_of_F_*Oo_X}) $\G$ is the right dual object for $\Ee$ as in Definition \ref{def:left-right_dual_exceptional_collections}.

\vspace*{0.2cm}

\item [(4)] The terms of ${\mathbb G}$ are, up to a shift, vector bundles on $X$.

\vspace*{0.2cm}

\item [(5)] Conversely, assume given a decomposition of ${{\sf F}_n}_{\ast}\Oo _X$ into the direct sum of vector bundles that form a full exceptional collection ${\mathbb G}=(\mathbb G _{-d+1},\dots ,\mathbb G _0)$ in $\Dd ^b(X)$. Then the multiplicity space at $\G \in \mathbb G _{i}$ is given by ${\rm H}^i(X,{\sf F}_n^{\ast}\Ee)$, where $\Ee$ is the right dual object for $\G$.

\end{itemize}

\end{theorem}

\vspace*{0.5cm}

\section{The incidence varieties in types ${\bf A}_3$ and ${\bf A}_4$}\label{sec:adjoint_var_rank3}

\vspace*{0.5cm}

Recall some notation from \cite[Section 6.1]{FrobflagvarI}.

\begin{definition}\label{def:Psi_bundles}
Given a semisimple algebraic group ${\bf G}$ and a dominant weight $\omega$ of $\bf G$, for $1\leq i\leq l$, where $l={\rm dim} (\nabla _{\omega})$ the bundle $\Psi ^{\omega}_i$ is set to be the pull--back of $\Omega _{\Pp (\nabla _{\omega})}^i(i)$ along the morphism ${\bf G}/{\bf B}\rightarrow \Pp (\nabla _{\omega})$ defined by a (semi)--ample line bundle $\Ll _{\omega}$. 
\comment{
The latter morphism factors through the projection $\pi _k: {\bf G}/{\bf B}\rightarrow {\bf G}/{\bf P}_{\hat \alpha _k}$ and the embedding ${\bf G}/{\bf P}_{\hat \alpha}\hookrightarrow \Pp (\nabla _{\omega _k})$, where ${\bf P}_{\hat \alpha _k}$ is the maximal parabolic subgroup of $\bf G$ obtained by adjoining all but one the simple roots to $\bf B$; the deleted simple root is the one with $\langle \omega _k , \alpha _k^{\vee}\rangle =1$.
}
\end{definition}

By definition, the bundles $\Psi _1^{\omega}$ fit into short exact sequences:

\vspace*{0.2cm}

\begin{equation}\label{eq:defining_sequence_for_Psi}
0\rightarrow \Psi _1^{\omega}\rightarrow \nabla _{\omega}\otimes \Oo _{{\bf G}/{\bf B}}\rightarrow 
\Ll _{\omega}\rightarrow 0,
\end{equation}

\vspace*{0.2cm}


Given a vector space $\sf V$ of dimension $n$ over $\sf k$, the incidence variety $X_n$\footnote{Also called the adjoint variety as being isomorphic to the orbit of the highest weight vector in the adjoint representation of ${\bf SL}_n$.} is defined to be the variety of partial flags of type $(1,n-1)$ in $\sf V$. It is a partial flag variety of the group ${\bf SL}_n$ with the Picard group isomorphic to ${\mathbb Z}^2$. Its canonical sheaf $\omega _{X_n}$ is isomorphic to $\Ll _{-(n-1)(\omega _1+\omega _{n-1})}$, where $\omega _1,\omega _2,\dots, \omega _{n-1}$ are the fundamental weights of ${\bf SL}_n$.

In this section we work out in detail the groups ${\bf SL}_4$ and ${\bf SL}_5$.
Thus, starting with ${\bf SL}_4$, denote $\omega _1,\omega _2,\omega _3$ the fundamental 
weights, and let $\alpha _1,\alpha _2,\alpha _3$ be the simple roots. For each simple root $\alpha _i$ let ${\bf P}_{\hat \alpha _i}\supset \bf B$ denote the corresponding minimal parabolic subgroup. The homogeneous spaces ${\bf SL}_4/{\bf P}_{\hat \alpha _i}$ can then be identified 
with varieties of partial flags $0\subset {\sf V}_1\subset {\sf V}_{i-1}\subset {\sf V}_{i+1}\subset {\sf V}$. There are the tautological bundles $\Uu _i$ for $i=1,2,3$ on ${\bf SL}_4/{\bf B}$.

The homogeneous space ${\bf SL}_4/{\bf P}_{\hat \alpha _2}$ can be identified with the partial flag variety ${\rm F}_{1,3,4}=: X_4$. 
The line bundles $\Ll _{\omega _1},\Ll _{\omega _3}$ and $\Ll _{\omega _1+\omega _3}$ on $X_4$ give rise to morphisms to $\Pp (\nabla _{\omega _1}),\Pp (\nabla _{\omega _3})$, and $\Pp (\nabla _{\omega _1+\omega _3})$, respectively. As in (\ref{eq:defining_sequence_for_Psi}), 
related to these morphisms are the short exact sequences:

\vspace*{0.2cm}

\begin{equation}\label{eq:Psi_omega_1}
0\rightarrow \Psi _1^{\omega _1}\rightarrow \nabla _{\omega _1}\otimes \Oo \rightarrow \Ll _{\omega _1}\rightarrow 0,
\end{equation}

\vspace*{0.2cm}

\begin{equation}\label{eq:Psi_omega_3}
0\rightarrow \Psi _1^{\omega _3}\rightarrow \nabla _{\omega _3}\otimes \Oo \rightarrow \Ll _{\omega _3}\rightarrow 0,
\end{equation}

\vspace*{0.2cm}

and 

\vspace*{0.2cm}

\begin{equation}\label{eq:Psi_omega_{1+3}}
0\rightarrow \Psi _1^{\omega _1+\omega _3}\rightarrow \nabla _{\omega _1+\omega _3}\otimes \Oo \rightarrow \Ll _{\omega _1+\omega _3}\rightarrow 0.
\end{equation}

\vspace*{0.2cm}

Further, denote $\Ee$ the quotient bundle $\Uu _3/\Uu _1$. Identifying 
the bundle $\Uu _3$ with $\Psi _1^{\omega _3}$ and $\Uu _1$ with $\Ll _{-\omega _1}$, we obtain a short exact sequence:

\vspace*{0.2cm}

\begin{equation}\label{eq:defseqforE}
0\rightarrow \Ll _{-\omega _1}\rightarrow \Psi _1^{\omega _3}\rightarrow \Ee \rightarrow 0.
\end{equation}

\vspace*{0.2cm}

Denoting $\pi:{\bf SL}_4/{\bf B}\rightarrow X_4$ the projection, one also obtains the bundle $\Ee$ as an extension:

\vspace*{0.2cm}

\begin{equation}\label{eq:seqforEonSL_4/B}
0\rightarrow \Ll _{\omega _1-\omega _2}\rightarrow \pi ^{\ast}\Ee \rightarrow \Ll _{\omega _2-\omega _3}\rightarrow 0.
\end{equation}

\vspace*{0.2cm}

We first going to produce a semiorthogonal decomposition of $\Dd ^b(X_4)$ that will satisfy the conditions of Theorem \ref{th:Frobdecomposclaim}. To this end, consider the semiorthogonal decomposition of $\Dd ^b(X_4) = \langle {\tilde \A}_{-1}, {\tilde \A}_0, \tilde \A , \A _0, \A _1,\A _2 \rangle$ \ given by the following block structure (the fact that it is indeed a semiorthogonal decomposition will be proven in Lemma \ref{lem:F-collection_adjoint_variety_SL_4/B} below):

\begin{figure}[H]
$$
\xymatrix{
&{\tilde \A}_{-1}& {\tilde \A}_0& {\tilde \A} & \A _0& \A _1& \A _2 & \\
& || & || & || & || & || & || & \\
&*++<10pt>[F]\txt{$\Ll _{-2\omega _1-\omega _3}$ \\ \\  $\Ll _{-\omega _1-2\omega _3}$}
&*++<10pt>[F]\txt{$\Ll _{-\omega _1-\omega _3}$}
&*++<10pt>[F]\txt{$\Ll _{-\omega _1}$  \\ \\  $\Ee \otimes \Ll _{-\omega _1}$ \\ \\ $\Ll _{-\omega _3} $}
 &*++<10pt>[F]\txt{$\Oo _{X_4}$}
 &*++<10pt>[F]\txt{$\Ll _{\omega _1}$ \\ \\ $\Ll _{\omega _3}$}
 &*++<10pt>[F]\txt{$\Ll _{2\omega _1}$  \\ \\ $\Ll _{\omega _1+\omega _3}$ \\ \\ $\Ll _{2\omega _3}$}
}
$$
 \end{figure}

\vspace{0.2cm}

We then consequently mutate the blocks $\A _1$ (resp., $\A _2$) to the left through the block $\A_ 0$ (resp., through the subcategory $\langle \A _0,\A_1\rangle$), while mutating the block ${\tilde \A}_{-1}$ to the right through ${\tilde \A}_0$ and leaving the block $\tilde \A$ intact to obtain the following decomposition:

\vspace{0.2cm}

\begin{figure}[H]
\begin{equation}\label{eq:F-decomposition_on_X_4}
\xymatrix{
&\C _{-5}& \C _{-4} & \C _{-3} & \C _{-2}& \C _{-1}& \C _0& \\
& || & || & || & || & || & || & \\
&*++<5pt>[F]\txt{$\Ll _{-\omega _1-\omega _3}$}
&*++<8pt>[F]\txt{$(\Psi _1^{\omega _1})^{\ast}\otimes \Ll _{-\omega _1-\omega _3} $ \\ \\  $(\Psi _1^{\omega _3})^{\ast}\otimes \Ll _{-\omega _1-\omega _3}$}
&*++<6pt>[F]\txt{$\Ll _{-\omega _1}$  \\ \\  $\Ee \otimes \Ll _{-\omega _1}$ \\ \\ $\Ll _{-\omega _3}$}
&*++<8pt>[F]\txt{$\Psi _2^{\omega _1}$  \\ \\  $\Psi _2^{{\omega _1,\omega _3}}$ \\ \\ $\Psi _2^{\omega _3}$}
 &*++<8pt>[F]\txt{$\Psi _1^{\omega _1}$ \\ \\ $\Psi _1^{\omega _3}$}
 &*++<5pt>[F]\txt{$\Oo _{X_4}$}
 }
 \end{equation}
 \end{figure}

\vspace{0.2cm}

Here $\Psi _2^{{\omega _1,\omega _3}}$ is the result of the left mutation of $\Ll _{\omega _1+\omega _3}$ through the subcategory $\langle \A _0,\A_1\rangle$

\vspace{0.2cm}

\begin{equation}
\Psi _2^{{\omega _1,\omega _3}}: = {\bf L}_{\langle \A _0,\A_1\rangle}(\Ll_{\omega _1+\omega _3}).
\end{equation}

\vspace{0.2cm}

\begin{remark}
{\rm The automorphism of ${\bf SL}_4$ interchanging the two simple roots $\alpha _1$ and $\alpha _3$ induces an automorphism of $X_4$. Hence the sought--for collection should be invariant under this automorphism as well. Since $\Ee \otimes \Ll _{-\omega _1}=\Ee ^{\ast}\otimes \Ll _{-\omega _3}$, the above collection is indeed invariant.}
\end{remark}

\begin{lemma}\label{lem:F-collection_adjoint_variety_SL_4/B}
Let $p>2$. Then the collection of subcategories $\C = \langle \C _{-5},\C _{-4}, \C _{-3}, \C _{-2}, \C _{-1}, \C _0 \rangle$ is a semiorthogonal decomposition of $\Dd ^b(X_4)$ satisfying the conditions of Theorem \ref{th:Frobdecomposclaim}. 
\end{lemma}

\begin{proof}
The proof is broken up into a few separate statements which are found below.
\end{proof}

\subsubsection{\bf Orthogonality}
We first observe that since the collection $\C$ is obtained by mutating the collection $\A  = \langle {\tilde \A}_{-1}, {\tilde \A}_0, \tilde \A , \A _0, \A _1,\A _2 \rangle$, the semiorthogonality of $\C$ is equivalent to that of the collection $\A$. The necessary orthogonalities between the line bundles in ${\tilde \A}$ follow immediately from Theorems \ref{th:Kempf} and \ref{th:AndthII}. To ensure the necessary orthogonalities for the bundle $\Ee \otimes \Ll _{-\omega _1}$, use sequences (\ref{eq:defseqforE}) and (\ref{eq:seqforEonSL_4/B}), and once again Theorems \ref{th:Kempf} and \ref{th:AndthII}.

\subsubsection{\bf Fullness}
The variety $X_4$ is a projective bundle over $\Pp (\nabla _{\omega _3})$ that is canonically isomorphic to $\Pp (\Psi ^1_{\omega _3})$ with relative Picard group being generated by $\Ll _{-\omega _1}$ (cf. sequence (\ref{eq:defseqforE})). By Theorem \ref{th:Orvlovth}, it is sufficient to prove that the full triangulated subcategory generated by $\A$ contains the subcategories $\Dd ^b(\Pp (\nabla _{\omega _3}))\otimes \Ll _{-2\omega _1},\Dd ^b(\Pp (\nabla _{\omega _3}))\otimes \Ll _{-\omega _1}$, and $\Dd ^b(\Pp (\nabla _{\omega _3}))$. It is clear that $\Dd ^b(\Pp (\nabla _{\omega _3}))$ is contained in $\A $. Tensoring sequences (\ref{eq:defseqforE}) and  (\ref{eq:Psi_omega_3}) with $\Ll _{-\omega _1}$, one obtains:

\vspace*{0.2cm}

\begin{equation}\label{eq:1st presentation for E}
0\rightarrow \Ll _{-2\omega _1}\rightarrow \Psi _1^{\omega _3}\otimes \Ll _{-\omega _1}\rightarrow \Ee \otimes \Ll _{-\omega _1}\rightarrow 0, 
\end{equation}

\vspace*{0.2cm}

and

\vspace*{0.2cm}

\begin{equation}\label{eq:2nd presentation for E}
0\rightarrow \Psi _1^{\omega _3}\otimes\Ll _{-\omega _1} \rightarrow \nabla _{\omega _3}\otimes \Ll _{-\omega _1} \rightarrow \Ll _{-\omega _1+\omega _3}\rightarrow 0,
\end{equation}

\vspace*{0.2cm}

Considering the resolution 

\vspace*{0.2cm}

\begin{equation}\label{eq:Psi _2 omega _3 Koszul res}
0\rightarrow \Psi _2^{\omega _3}\rightarrow \nabla _{\omega _2}\otimes \Oo \rightarrow \nabla _{\omega _1}\otimes \Ll _{\omega _3}\rightarrow \Ll _{2\omega _3}\rightarrow 0,
\end{equation}

\vspace*{0.2cm}

and tensoring it with $\Ll _{-\omega _1}$, we obtain that $\Ll _{2\omega _3-\omega _1}$ is in $\A $, in virtue of an isomorphism $\Psi _2^{\omega _3}\otimes \Ll _{-\omega _1}=(\Psi _1^{\omega _3})^{\ast}\otimes \Ll _{-\omega _1-\omega _3}$. Hence, 
$\Dd ^b(\Pp (\nabla _{\omega _1}))\otimes \Ll _{-\omega _1}$ is also contained in $\A $.

Clearly, $\Ll _{-2\omega _1-\omega _3}$ and $\Ll _{-2\omega _1}$ are in $\A $. Considering the short exact sequence 

\vspace*{0.2cm}

\begin{equation}
0\rightarrow \Ll _{-2\omega _1+\omega _3}\rightarrow \Psi _1^{\omega _3}\otimes \Ll _{-\omega _1+\omega _3}\rightarrow \Ee ^{\ast}\rightarrow 0,
\end{equation}

\vspace*{0.2cm}

that is obtained from (\ref{eq:1st presentation for E}) by tensoring with $\Ll _{\omega _3}$ and using an isomorphism $\Ee ^{\ast}=\Ee \otimes \Ll _{-\omega _1+\omega _3}$. Now $\Ee ^{\ast}$ is in $\A$; this is seen by taking the dual of (\ref{eq:defseqforE}). Taking into account that 
$\Psi _1^{\omega _3}\otimes \Ll _{-\omega _1+\omega _3}\in \langle \Ll _{-\omega _1+\omega _3},\Ll _{-\omega _1+2\omega _3}\rangle$, 
we conclude that $\Ll _{-2\omega _1+\omega _3}$ is in $\A $.  Finally, assuming that there is a non--trivial object in the right orthogonal to $\A $, we see that it must be quasiisomorphic to $\Ll _{-2\omega _1-2\omega _3}\otimes {\sf V}^{\bullet}$ for some graded vector space ${\sf V}^{\bullet}$. However, 
$\Hom ^{\bullet}(\Psi _1^{\omega _1+\omega _3},\Ll _{-2\omega _1-2\omega _3}\otimes {\sf V}^{\bullet})={\sf V}^{\bullet}[5]$,
since $\omega _{X_4}=\Ll _{-3\omega _1-3\omega _3}$. Hence, ${\sf V}^{\bullet}=0$ and the statement follows.

\subsubsection{\bf Cohomology of Frobenius pull--backs}

Let us now check the property ${\rm H}^j(X_4,{\sf F}^{\ast}(?))=0$ for $j\neq i$ for a bundle $?\in  \C _{-i}, i=0,\dots,5$. The non--trivial verifications here concern the two bundles in the block $\C _{-4}$, the bundle $\Ee \otimes \Ll _{-\omega _1}$ in $\C _{-3}$, and $\Psi _2^{{\omega _1,\omega _3}}$ in $\C _{-2}$.

\begin{claim}
One has ${\rm H}^i(X_4,{\sf F}^{\ast}\Psi _2^{{\omega _1,\omega _3}})=0$ for $i\neq 2$.
\end{claim}

\begin{proof}
Recall that $\Psi _2^{{\omega _1,\omega _3}}$ is defined to be the left mutation of $\Psi _1^{\omega _1+\omega _3}$ through the subcategory $\langle  \Psi _1^{\omega _1}, \Psi _1^{\omega _3}\rangle$. Thus, there is a resolution:

\vspace*{0.2cm}

\begin{equation}
0\rightarrow \Psi _2^{{\omega _1,\omega _3}}\rightarrow \Psi _1^{\omega _1}\otimes \nabla _{\omega _3}\oplus \Psi _1^{\omega _3}\otimes \nabla _{\omega _1}\rightarrow \nabla _{\omega _1+\omega _3}\otimes \Oo \rightarrow \Ll _{\omega _1+\omega _3}\rightarrow 0.
\end{equation}

\vspace*{0.2cm}

Clearly, ${\rm H}^i(X_4,{\sf F}^{\ast}\Psi _1^{\omega _k})=0$ for $i\neq 1$ and $k=1,3$ (see \cite[Proposition 6.1]{FrobflagvarI}). 
Thus, one obtains ${\rm H}^i(X_4,{\sf F}^{\ast}\Psi _2^{{\omega _1,\omega _3}})=0$ for $i\neq 1,2$.
\comment{
Further,
the map ${\sf F}^{\ast}(\nabla _{\omega _1+\omega _3})\rightarrow \nabla _{p(\omega _1+\omega _3})$ is injective \footnote{E.g., from the linkage principle for $p>3$; check/reference for $p\leq 3$.}.}

It follows from the subsequent Proposition \ref{prop:right_dual_dec_F-collection_n=4} and Claims \ref{cl:A_-3} and \ref{cl:A_-4} that ${\rm H}^1(X_4,{\sf F}^{\ast}\Psi _2^{{\omega _1,\omega _3}})=0$. 
Indeed, by Lemma \ref{lem:F-collection_adjoint_variety_SL_4/B} and Theorem \ref{th:Kuzbasechange} one obtains, upon identifying the right dual collection with (\ref{eq:right_dual_F-decomposition_on_X_4}), a resolution of the diagonal of $X_4$. With this resolution of the diagonal in hand, one obtains a left resolution of $\Psi _2^{{\omega _1,\omega _3}}$. Specifically, denote $p_1,p_2$ the two projections of $X_4\times X_4$ onto $X_4$, and tensor the resolution of the digonal along the right $\boxtimes$ factor with $\Ll _{-2(\omega _1+\omega _3)}$. Pushing forward the result onto $X_4$ along $p_1$, one obtains:

\vspace*{0.2cm}

\begin{eqnarray}
& 0\rightarrow \Ll _{-2(\omega _1+\omega _3)}\rightarrow \Ll _{-(\omega _1+\omega _3)}\otimes \Delta _{\omega _1+\omega _3}\rightarrow  \\
& (\Psi _1^{\omega _1})^{\ast}\otimes \Ll _{-\omega _1-\omega _3}\otimes \Delta _{\omega _3}\oplus (\Psi _1^{\omega _3})^{\ast}\otimes \Ll _{-\omega _1-2\omega _3}\otimes \Delta _{\omega _1}\rightarrow 
\nabla _{\omega _2}\otimes {\tilde \G}^{\ast}\otimes \Ll _{-2\omega _1-2\omega _3}\rightarrow \Psi _2^{\omega _1,\omega _3}\rightarrow 0 \nonumber 
\end{eqnarray}

\vspace*{0.2cm}

(see (\ref{eq:right_dual_E(-omega_1)}) for the definition of the bundle 
${\tilde \G}$). It also follows straightforwardly from (\ref{eq:right_dual_E(-omega_1)}) that ${\rm H}^i(X_4,{\sf F}^{\ast}({\tilde \G}^{\ast}\otimes \Ll _{-2\omega _1-2\omega _3}))=0$ for $i<3$. Combining altogether all the above facts, one concludes that
${\rm H}^1(X_4,{\sf F}^{\ast}\Psi _2^{{\omega _1,\omega _3}})=0$. 

\end{proof}

\begin{claim}\label{cl:A_-3} 
One has ${\rm H}^i(X_4,{\sf F}^{\ast}(\Ee \otimes \Ll _{-\omega _1}))=0$ for $i\neq 3$.
\end{claim}

\begin{proof}
Recall the short exact sequences:

\vspace*{0.2cm}

\begin{equation}
0\rightarrow \Ll _{-\omega _2}\rightarrow \pi ^{\ast}\Ee \otimes \Ll _{-\omega _1} \rightarrow \Ll _{-\omega _1+\omega _2-\omega _3}\rightarrow 0,
\end{equation}

\vspace*{0.2cm}

and,

\begin{equation}
0\rightarrow \Ll _{-2\omega _1}\rightarrow \Psi _1^{\omega _3}\otimes \Ll _{-\omega _1}\rightarrow \Ee \otimes \Ll _{-\omega _1}\rightarrow 0.
\end{equation}

\vspace*{0.2cm}

Finally, tensoring (\ref{eq:Psi_omega_3}) with $\Ll _{-\omega _1+\omega _3}$, obtain 

\vspace*{0.2cm}

\begin{equation}
0\rightarrow \Psi _1^{\omega _3}\otimes \Ll _{-\omega _1}\rightarrow \nabla _{\omega _3}\otimes \Ll _{-\omega _1}\rightarrow \Ll _{-\omega _1+\omega _3}\rightarrow 0.	
\end{equation}

\vspace*{0.2cm}

From the last sequence we conclude that ${\rm H}^i(X_4,{\sf F}^{\ast}(\Psi _1^{\omega _3}\otimes \Ll _{-\omega _1}))=0$ for $i\neq 3$. Indeed, ${\rm H}^i(X_4,\nabla _{\omega _3}\otimes \Ll _{-p\omega _1})=0$ for $i\neq 3$ (for $p=2,3$ it is acyclic). On the other hand, one has ${\rm H}^i(X_4,\Ll _{-p\omega _1+p\omega _3})=0$ for $i\neq 2$. Indeed, $s_{\alpha _2}\cdot s_{\alpha _1}\cdot (-p\omega _1+p\omega _3)=(p-3)\omega _2+2\omega _3$ (for $p=2$ it is acyclic). Therefore, from the middle sequence we obtain that ${\rm H}^i(X_4,{\sf F}^{\ast}(\Ee \otimes \Ll _{-\omega _1}))=0$ for $i\neq 2,3$.
On the other hand, from the first sequence we see that ${\rm H}^i(X_4,{\sf F}^{\ast}(\Ee \otimes \Ll _{-\omega _1}))=0$ for $i\neq 3,4$ (we use throughout the fact that computing the cohomology of a coherent sheaf on $X_4$ is equivalent to doing that on the full flag variety ${\bf SL}_4/{\bf B}$, the functor $\pi ^{\ast}: \Dd ^b(X_4)\rightarrow \Dd ^b({\bf SL}_4/{\bf B})$ being fully faithful). Indeed, ${\rm H}^i({\bf SL}_4/{\bf B},\Ll _{-p\omega _2})=0$ for $i\neq 4$ (for $p=2,3$ it is acyclic), while 
 ${\rm H}^i({\bf SL}_4/{\bf B},\Ll _{-p\omega _1+p\omega _2-\omega _3})=0$ for $i\neq 3$ (for $p=2$ it is trivial for $i\neq 2$, and isomorphic to ${\sf k}$ in this degree, and for $p=3$ it is acyclic): indeed, $s_{\alpha _2}\cdot s_{\alpha _3}\cdot s_{\alpha _1}\cdot (-p\omega _1+p\omega _2-\omega _3)=\omega _1+(p-4)\omega _2+\omega _3$.
\end{proof}

\begin{claim}\label{cl:A_-4}
One has ${\rm H}^i(X_4,{\sf F}^{\ast}((\Psi _1^{\omega _1})^{\ast}\otimes \Ll _{-\omega _1-\omega _3}))=
{\rm H}^i(X_4,{\sf F}^{\ast}((\Psi _1^{\omega _3})^{\ast}\otimes \Ll _{-\omega _1-\omega _3}))=0$ for $i\neq 4$.
\end{claim}

\begin{proof}
It is sufficient to prove the statement for the first group. On the one hand, one has a short exact sequence (note that $(\Psi _1^{\omega _1})^{\ast}\otimes \Ll _{-\omega _1-\omega _3} = \T _{\Pp (\nabla _{\omega _1})}\otimes \Ll _{-2\omega _1-\omega _3}$): 

\vspace*{0.2cm}

\begin{equation}
0\rightarrow  \Ll _{-2\omega _1-\omega _3}\rightarrow \nabla _{\omega _1}\otimes \Ll _{-\omega _1-\omega _3}\rightarrow 
(\Psi _1^{\omega _1})^{\ast}\otimes \Ll _{-\omega _1-\omega _3}\rightarrow 0,
\end{equation}

\vspace*{0.2cm}

from which we conclude that ${\rm H}^i(X_4,{\sf F}^{\ast}((\Psi _1^{\omega _1})^{\ast}\otimes \Ll _{-\omega _1-\omega _3}))=0$ for $i\neq 4,5$. On the other hand, we have an isomorphism $(\Psi _1^{\omega _1})^{\ast}\otimes \Ll _{-\omega _1}=\Psi _2^{\omega _1}$ and the resolution

\vspace*{0.2cm}

\begin{equation}\label{eq:Psi _2 omega _1 Koszul res}
0\rightarrow \Psi _2^{\omega _1}\rightarrow \nabla _{\omega _2}\otimes \Oo \rightarrow \nabla _{\omega _1}\otimes \Ll _{\omega _1}\rightarrow \Ll _{2\omega _1}\rightarrow 0.
\end{equation}

\vspace*{0.2cm}

Tensoring it with $\Ll _{-\omega _3}$, one obtains:

\vspace*{0.2cm}

\begin{equation}
0\rightarrow (\Psi _1^{\omega _1})^{\ast}\otimes \Ll _{-\omega _1-\omega _3}\rightarrow \nabla _{\omega _2}\otimes \Ll _{-\omega _3} \rightarrow \nabla _{\omega _3}\otimes \Ll _{\omega _1-\omega _3}\rightarrow \Ll _{2\omega _1-\omega _3}\rightarrow 0.
\end{equation}

\vspace*{0.2cm}

We see that ${\rm H}^i(X_4, \nabla _{\omega _2}\otimes \Ll _{-p\omega _3})=0$ for $i\neq 3$ (for $p=2,3$ it is acyclic), and  ${\rm H}^i(X_4,\nabla _{\omega _3}\otimes \Ll _{p\omega _1-p\omega _3})=0$ for $i\neq 2$.
Finally, ${\rm H}^i(X_4,\Ll _{2p\omega _1-p\omega _3})=0$ for $i\neq 2$. Indeed, 
$s_{\alpha _2}\cdot s_{\alpha _3}\cdot (2p\omega _1-p\omega _3)=(p+2)\omega _1+(p-3)\omega _2$. Therefore, 
${\rm H}^5(X_4,{\sf F}^{\ast}((\Psi _1^{\omega _1})^{\ast}\otimes \Ll _{-\omega _1-\omega _3})=0$, and the statement follows.

\end{proof}

\begin{proposition}\label{prop:right_dual_dec_F-collection_n=4}

The right dual decomposition $\tilde C$ to (\ref{eq:F-decomposition_on_X_4}) consists of the following subcategories:

\begin{figure}[H]
\begin{equation}\label{eq:right_dual_F-decomposition_on_X_4}
\xymatrix{
&{\tilde \C}_0 & {\tilde \C}_1 & {\tilde \C}_2 & {\tilde \C}_3 & {\tilde \C}_4 & {\tilde \C}_5& \\
& || & || & || & || & || & || & \\
&*++<5pt>[F]\txt{$\Oo _{X_4}$}
&*++<8pt>[F]\txt{$\Ll _{\omega _1} $ \\ \\  $\Ll _{\omega _3}$}
&*++<6pt>[F]\txt{$\Ll _{2\omega _1}$  \\ \\  $\Ll _{\omega _1+\omega _3}$ \\ \\ $\Ll _{2\omega _3}$}
&*++<8pt>[F]\txt{$\Psi _2^{\omega _1}\otimes \Ll _{\omega _1+2\omega _3}$ \\ \\   ${\tilde \G}$ \\ \\ $\Psi _2^{\omega _3}\otimes \Ll _{2\omega _1+\omega _3}$}
 &*++<8pt>[F]\txt{$\Ll _{\omega _1+2\omega _3}$ \\ \\ $\Ll _{2\omega _1+\omega _3}$}
 &*++<5pt>[F]\txt{$\Ll _{2(\omega _1+\omega _3)}$}
 }
 \end{equation}
 \end{figure}

\vspace{0.2cm}

where ${\tilde \G}$ is a vector bundle of rank 4 fitting into a unique non--split short exact sequence:

\vspace{0.2cm}

\begin{equation}\label{eq:right_dual_E(-omega_1)}
0\rightarrow \Ll _{2\omega _3}\rightarrow {\tilde \G}\rightarrow \Psi _1^{\omega _3}\otimes \Ll _{2\omega _1+\omega _3}\rightarrow 0.
\end{equation}

\vspace{0.2cm}

\end{proposition}

\begin{proof}
The right dual objects to all the generators of $\A$ but $\Ee \otimes \Ll _{-\omega _1}$ are calculated rather straightforwardly.
To compute the right dual to $\Ee \otimes \Ll _{-\omega _1}$, 
it is sufficient to compute its left mutation through the subcategory generated by $\langle {\tilde \A}_{-1},{\tilde \A}_0\rangle$ and then use Lemma \ref{lem:mutations_canonical_class}.
 One computes $\Hom ^{\bullet}(\Ll _{-\omega _1-\omega _3},\Ee \otimes \Ll _{-\omega _1})=\nabla _{\omega _2}$. One then  obtains the commutative diagram depicted below:

\vspace{0.2cm}

\begin{figure}[H]
$$
\xymatrix{
&  & 0 \ar@{->}[d] & 0 \ar@{->}[d] & 0 \ar@{->}[d]&\\
& 0\ar@{->}[r]\ar@{->}[d] &\Psi _2^{\omega _1}\otimes \Ll _{-\omega _1-\omega _3} \ar@{->}[r]\ar@{->}[d] & \Ff \ar@{->}[r]\ar@{->}[d] & \Ll _{-2\omega _3}\ar@{->}[r]\ar@{->}[d] & 0\\
0\ar@{->}[r] & 0 \ar@{->}[r]\ar@{->}[d] & \nabla _{\omega _2}\otimes \Ll _{-\omega _1-\omega _3}  \ar@{->}[r]^{\simeq}\ar@{->}[d] & \nabla _{\omega _2}\otimes \Ll _{-\omega _1-\omega _3} \ar@{->}[d]\ar@{->}[r] & 0\ar@{->}[d] \\
0\ar@{->}[r] & \Ll _{-2\omega _3}\ar@{->}[r]\ar@{->}[d] & \Psi _1^{\omega _1}\otimes \Ll _{-\omega _3}\ar@{->}[r]\ar@{->}[d] & \Ee \otimes \Ll _{-\omega _1}\ar@{->}[r]\ar@{->}[d]& 0\\
& 0 & 0 & 0 & & \\
}
$$
 \end{figure}
  
\vspace{0.2cm}

Thus, ${\bf L}_{{\tilde \A}_0}(\Ee \otimes \Ll _{-\omega _1})=\Ff [1]$, where $\Ff$ is the extension from the top row of the above diagram that corresponds to a unique non--split extension $\Ext ^1(\Ll _{-2\omega _3},\Psi _2^{\omega _1}\otimes \Ll _{-\omega _1-\omega _3})={\sf k}$. Further, as $\Ll _{-2\omega _1-\omega _3}$ and $\Ll _{-\omega _1-2\omega _3}$ are mutually orthogonal, the left mutation of $\Ff$ through $\langle \Ll _{-2\omega _1-\omega _3},\Ll _{-\omega _1-2\omega _3}\rangle$ is found from the triangle 

\vspace{0.2cm}

\begin{equation}
\dots \rightarrow \nabla _{\omega _1}\otimes (\Ll _{-2\omega _1-\omega _3}\oplus \Ll _{-\omega _1-2\omega _3})\rightarrow \Ff \rightarrow \G [1]\rightarrow \dots 
\end{equation}

\vspace{0.2cm}

as one calculates  $\Hom ^{\bullet}(\Ll _{-2\omega _1-\omega _3},\Ff )= \Hom ^{\bullet}(\Ll _{-\omega _1-2\omega _3},\Ff )=\nabla _{\omega _1}$ from the same top row of the above diagram. Consider the diagram

\vspace{0.2cm}

\begin{figure}[H]
$$
\xymatrix{
& 0\ar@{->}[d] & 0 \ar@{->}[d] &  0 \ar@{->}[d] & &\\
0\ar@{->}[r] & \Ll _{-3\omega _1-\omega _3}\ar@{->}[r]\ar@{->}[d] &\G \ar@{->}[r]\ar@{->}[d] & \Psi _1^{\omega _1}\otimes \Ll _{-\omega _1-2\omega _3} \ar@{->}[r]\ar@{->}[d] & 0\\
0\ar@{->}[r] & \nabla _{\omega _1}\otimes \Ll _{-2\omega _1-\omega _3}\ar@{->}[r]\ar@{->}[d] &  \nabla _{\omega _1}\otimes (\Ll _{-2\omega _1-\omega _3}\oplus \Ll _{-\omega _1-2\omega _3}) \ar@{->}[r]\ar@{->}[d] & \nabla _{\omega _1}\otimes \Ll _{-\omega _1-2\omega _3}\ar@{->}[d]\ar@{->}[r] & 0 \\
0\ar@{->}[r] & \Psi _2^{\omega _1}\otimes \Ll _{-\omega _1-\omega _3}\ar@{->}[r]\ar@{->}[d] & \Ff \ar@{->}[r]\ar@{->}[d] &\Ll _{-2\omega _3}\ar@{->}[r]\ar@{->}[d] & 0\\
& 0 & 0 & 0 & & \\
}
$$
 \end{figure}
 
 \vspace{0.2cm}
  
The bottom row of the above diagram is the defining short exact sequence for $\Ff$, the middle one being the split extension. One thus finds the left mutation ${\bf L}_{{\tilde \A}_{-1}}(\Ff)$ to be isomorphic to $\G [1]$, where $\G$ is defined by the top row and is obtained as a unique non--split extension corresponding to $\Ext ^1(\Psi _1^{\omega _1}\otimes \Ll _{-\omega _1-2\omega _3}, \Ll _{-3\omega _1-\omega _3})={\sf k}$. Finally, the right dual to $\Ee \otimes \Ll _{-\omega _1}$ is seen to be isomorphic to

\vspace{0.2cm}

\begin{equation}
{\bf R}_{\langle \A _{0}\A _{1},\A _2\rangle}(\Ee \otimes \Ll _{-\omega _1})=
{\bf L}_{\langle {\tilde \A }_{-1},{\tilde \A}_{0}\rangle}(\Ee \otimes \Ll _{-\omega _1})\otimes \Ll _{3(\omega _1+\omega _3)} = {\tilde \G}[2-5] ={\tilde \G}[-3] .
\end{equation}

\vspace{0.2cm}

For consistency of the above calculations, one computes $\Hom ^{\bullet}({\tilde \G}, \Ee \otimes \Ll _{-\omega _1})={\sf k}[-3]$, while 
$\Hom ^{\bullet}({\tilde \G} , -)=0$ for all other exceptional generators of the decomposition $\C$, thus verifying by Proposition \ref{prop:dual_exc_coll_characterization} that the object ${\tilde \G}[-3]$ is the right dual to $\Ee \otimes \Ll _{-\omega _1}$.

\end{proof}

One thus obtains:

\begin{theorem}\label{th:F-theorem_X_4}
The bundle ${\sf F}_{\ast}\Oo _{X_4}$ decomposes into the direct sum of vector bundles with indecomposable summands being isomorphic to:

\vspace*{0.2cm}

\begin{eqnarray}\label{eq:Frobrightdualcoll_X_4}
&  \Oo _{X_4}, \quad \Ll _{-\omega _1}, \quad \Ll _{-\omega _3},  \quad   \Ll _{-2\omega _1}, \quad \Ll _{-\omega _1-\omega _3}, \quad\Ll _{-2\omega _3},\\
& (\Psi _2^{\omega _1})^{\ast}\otimes \Ll _{-\omega _1-2\omega _3}, \quad {\tilde \G}^{\ast}, \quad(\Psi _2^{\omega _3})^{\ast}\otimes \Ll _{-2\omega _1-\omega _3},
\quad \Ll _{-\omega _1-2\omega _3}, \quad \Ll _{-2\omega _1-\omega _3}, \quad \Ll _{-2(\omega _1+\omega _3)} \nonumber
\end{eqnarray}

\vspace*{0.2cm}

The multiplicity spaces at each indecomposable summand are isomorphic, respectively, to  the cohomology of the Frobenius pull--back ${\sf F}^{\ast}$ applied to the right dual of the given summand that is found from (\ref{eq:F-decomposition_on_X_4}).
\end{theorem}

\subsection{Type ${\bf A}_4$}

\vspace{0.5cm}

Following the previous section, consider

$$
\xymatrix{
{\tilde \A}_{-2} & {\tilde \A}_{-1} & {\tilde \A}_0& \tilde \A & \A _0 & \A _1& \A _2&  \A _3 \\
||  & ||  & ||  & || & || & || & || & || & || & || &\\
*+[F]\txt{$\Ll _{-3\omega _1-\omega _4}$ \\ \\ $\Ll _{-2\omega _1-2\omega _4}$  \\ \\   $\Ll _{-\omega _1-3\omega _4}$} 
&*+[F]\txt{$\Ll _{-2\omega _1-\omega _4}$ \\ \\ $\Ll _{-\omega _1-2\omega _4}$}
&*+[F]\txt{$\Ll _{-\omega _1-\omega _4}$}
&*+[F]\txt{$\Ll _{-\omega _1}$  \\ \\ $\Ee \otimes \Ll _{-\omega _1}$ \\ \\ $\Lambda ^2\Ee \otimes \Ll _{-\omega _1}$ \\ \\ $\Ll _{-\omega _4} $}
&*+[F]\txt{$\Oo _{X_4}$}
&*+[F]\txt{$\Ll _{\omega _1}$ \\ \\  $\Ll _{\omega _4}$}
&*+[F]\txt{$\Ll _{2\omega _1}$  \\ \\ $\Ll _{\omega _1+\omega _4}$ \\ \\  $\Ll _{2\omega _4}$}
&*+[F]\txt{$\Ll _{3\omega _1}$  \\ \\ $\Ll _{2\omega _1+\omega _4}$ \\ \\ $\Ll _{\omega _1+2\omega _4}$ \\ \\ $\Ll _{3\omega _4}$}
} \nonumber
$$

\vspace{0.5cm}

Likewise, we consequently mutate the blocks $\A _i$ for $i=1,2,3$ to the left through the subcategory generated by $\langle \A _0,\dots, \A _{i-1}\rangle$, while mutating the blocks ${\tilde \A}_{-i}$ for $i=1,2$ to the right through 
the subcategory generated by $\langle {\tilde \A}_{-i+1},\dots, {\tilde \A} _0\rangle$ obtaining the mutated block structure:

\vspace{0.5cm}

\begin{equation}\label{eq:F-decomposition_on_X_5}
\begin{tabular}{c}
\xymatrix{
\C _{-7} & \C _{-6} & \C _{-5}& \C _{-4} & \C _{-3}& \C _{-2} & \C _{-1} & \C _0 \\
||  & ||  & ||  & || & || & || & || & || &  \\
*+[F]\txt{$\Ll _{-\omega _1-\omega _4}$} 
&*+[F]\txt{$\Phi _1^{\omega _1}$ \\ \\ 
$\Phi _1^{\omega _4}$}
&*+[F]\txt{$\Phi _2^{\omega _1}$ \\ \\ $\Phi _2^{\omega _1,\omega _4}$ \\ \\ $\Phi _2^{\omega _4}$}
&*+[F]\txt{$\Ll _{-\omega _1}$  \\ \\ $\Ee \otimes \Ll _{-\omega _1}$ \\ \\ $\Lambda ^2\Ee \otimes \Ll _{-\omega _1}$ \\ \\ $\Ll _{-\omega _4}$}
&*+[F]\txt{$\Psi _3^{\omega _1}$ \\ \\ $\Psi _3^{2\omega _1,\omega _4}$ \\ \\$\Psi _3^{\omega _1,2\omega _4}$ \\ \\ $\Psi _3^{\omega _3}$}
&*+[F]\txt{$\Psi _2^{\omega _1}$ \\ \\ $\Psi _2^{\omega _1,\omega _4}$\\ \\ $\Psi _2^{\omega _4}$}
&*+[F]\txt{$\Psi _1^{\omega _1}$ \\ \\ $\Psi _1^{\omega _4}$}
&*+[F]\txt{$\Oo _{X_4}$}
}
\end{tabular}
\end{equation}

\vspace{0.5cm}

where $\Psi _{k}^{i\omega _1,(k-i)\omega _4}$ 
for $k=1,2,3$ and $0\leq i\leq k$ 
denotes the left mutation of $\Ll _{i\omega _1+(k-i)\omega _4}$ through $\langle \A _0,\dots, \A _{k-1}\rangle$, while $\Phi _k^{i\omega _1,(k-i)\omega _4}$ for $k=1,2$ and $0\leq i\leq k$ (shortened to $\Phi _{k}^{\omega _1}$ or to $\Phi _{k}^{\omega _4}$ for $i=0$ or $i=k$)\footnote{Note that there are isomorphisms $\Psi _{k}^{k\omega _1,0}=\Psi _{k}^{\omega _1}$ and $
\Psi _{k}^{0,k\omega _4}=\Psi _{k}^{\omega _4}$ in the notation of Definition \ref{def:Psi_bundles} which justifies the notaion in (\ref{eq:F-decomposition_on_X_5}).} denotes, correspondingly, the right mutation of $\Ll _{(-k+i-1)\omega _1-(i+1)\omega _4}$ through $\langle {\tilde \A}_{-k+1},\dots, {\tilde \A}_0\rangle$
(in fact, there are isomorphisms $\Phi _k^{i\omega _1,(k-i)\omega _4}=(\Psi _i^{i\omega _1,(k-i)\omega _4})^{\ast}\otimes \Ll _{-\omega _1-\omega _4}$ for $k=1,4$ and $0\leq i\leq k$).

\vspace{0.2cm}

Similarly, one has the following 

\begin{lemma}\label{lem:F-collection_adjoint_variety_SL_5/B}
Let $p>2$. Then the collection of subcategories $\C = \langle \C _{-7},\C _{-6}, \dots ,\C _0 \rangle$ as above is a semiorthogonal decomposition of $\Dd ^b(X_5)$ satisfying the conditions of Theorem \ref{th:Frobdecomposclaim}. 
\end{lemma}

\begin{proof}
The proof is essentially analogous to that of Lemma \ref{lem:F-collection_adjoint_variety_SL_4/B}. The new feature is the appearance of the second exterior power of the bundle $\Ee$ whose Frobenius pull--back is supposed to have the single non--trivial cohomology in the prescribed degree (specifically, in degree $4$ in the considered case). More generally, one has:

\begin{claim}\label{cl:Lambda ^i(Ee)}
One has ${\rm H}^k(X_n,{\sf F}^{\ast}(\Lambda ^i\Ee\otimes \Ll _{-\omega _1}))=0$ for $k\neq n-1$ and $i=1,n-3$.
\end{claim}

\begin{proof}
Recall the defining short exact sequence for $\Ee\otimes \Ll _{-\omega _1}$:

\vspace{0.2cm}

\begin{equation}\label{eq:defseq_for_E(-omega_1)}
0\rightarrow \Ll _{-2\omega _1}\rightarrow \Psi _1^{\omega _{n-1}}\otimes \Ll _{-\omega _1}\rightarrow \Ee\otimes \Ll _{-\omega _1}\rightarrow 0.
\end{equation}

\vspace{0.2cm}

Considering the short exact sequence 

\vspace{0.2cm}

\begin{equation}
0\rightarrow \Psi _1^{\omega _n}\otimes \Ll _{-\omega _1}\rightarrow \nabla _{\omega _1}\otimes \Ll _{-\omega _1}\rightarrow \Ll _{-\omega _1+\omega _{n+1}}\rightarrow 0,
\end{equation}

\vspace{0.2cm}

one sees that ${\rm H}^k(X_n,{\sf F}^{\ast}(\Psi _1^{\omega _n}\otimes \Ll _{-\omega _1}))=0$ for $k\neq n-1$. Indeed, one has ${\rm H}^k(X_n,\Ll _{-p\omega _1})={\rm H}^k(\Pp (\nabla _{\omega _1}),\Ll _{-p\omega _1})=0$ for $k\neq n-1$, while 

\vspace{0.2cm}

\begin{equation}\label{eq:cohomology_-p(omega _1-omega _n-1)}
s_{\alpha _{n-2}}\cdot s_{\alpha _{n-3}}\cdot  \dots \cdot s_{\alpha _1}\cdot (-p\omega _1 + p\omega _{n-1}) = (p-n)\omega _{n-2}+(n-1)\omega _{n-1}.
\end{equation}

\vspace{0.2cm}

Hence, ${\rm H}^k(X_n,\Ll _{-p\omega _1+p\omega _{n-1}})=0$ for $k\neq n-2$, and the statement follows. From sequence (\ref{eq:defseq_for_E(-omega_1)}) one then obtains 
${\rm H}^k(X_n,{\sf F}^{\ast}(\Ee\otimes \Ll _{-\omega _1}))=0$ for $k\neq n-2,n-1$. On the other hand, the bundle $\pi ^{\ast}\Ee\otimes \Ll _{-\omega _1}$ is filtered with the set of graded factors being isomorphic to

\vspace{0.2cm}

\begin{equation}\label{eq:filtration_on_E(-omega _1)}
\Ll _{-\omega _2},\Ll _{-\omega _1+\omega _2-\omega _3},\dots ,\Ll _{-\omega _1+\omega _{n-3}-\omega _{n-2}} .
\end{equation}

\vspace{0.2cm}

Using Theorems \ref{th:AndthII} and \ref{th:Kempf}, one checks that ${\rm H}^{n-2}(X_n,\Ll _{p\lambda})=0$, where $\lambda$ is any weight from the set (\ref{eq:filtration_on_E(-omega _1)}), and thus Claim \ref{cl:Lambda ^i(Ee)} for $i=1$ follows.

The case $i=n-3$ is similar to the above. One has $\Lambda ^{n-3}\Ee \otimes \Ll _{-\omega _1}=\Ee ^{\ast}\otimes {\rm det}(\Ee)\otimes \Ll _{-\omega _1}=\Ee ^{\ast}\otimes \Ll _{-\omega _{n-1}}$, since ${\rm det}(\Ee)=\Ll _{\omega _1-\omega _{n-1}}$. Consider the short exact sequence 

\vspace{0.2cm}

\begin{equation}
0\rightarrow \Ee ^{\ast}\otimes \Ll _{-\omega _{n-1}}\rightarrow (\Psi _1^{\omega _{n-1}})^{\ast}\otimes \Ll _{-\omega _{n-1}}\rightarrow \Ll _{\omega _1-\omega _{n+1}}\rightarrow 0,
\end{equation}

\vspace{0.2cm}

One obtains ${\rm H}^k(X_n,{\sf F}^{\ast}((\Psi _1^{\omega _{n-1}})^{\ast}\otimes \Ll _{-\omega _{n-1}}))=0$ for $k\neq n-2,n-1$, while (cf. (\ref{eq:cohomology_-p(omega _1-omega _n-1)}))
${\rm H}^k(X_n,\Ll _{p(\omega _1-\omega _{n+1})})=0$ for $k\neq n-2$. Thus, 
${\rm H}^k(X_n,{\sf F}^{\ast}(\Ee ^{\ast}\otimes \Ll _{-\omega _{n-1}}))=0$ 
 for $k\neq n-2,n-1$. Arguing as in (\ref{eq:filtration_on_E(-omega _1)}), one ensures that 
 ${\rm H}^{n-2}(X_n,{\sf F}^{\ast}(\Ee ^{\ast}\otimes \Ll _{-\omega _{n-1}}))=0$, hence the statement.
  
\end{proof}

\end{proof}

\subsubsection{The right dual collection}
The terms of right dual collection to (\ref{eq:F-decomposition_on_X_5}) are calculated as in Proposition \ref{prop:right_dual_dec_F-collection_n=4}; most of these can be obtained immediately from the construction. As for the bundles $\Ee \otimes \Ll _{-\omega _1}$ and $\Lambda ^2(\Ee \otimes \Ll _{-\omega _1})$, repeating the calculation from Proposition \ref{prop:right_dual_dec_F-collection_n=4}, one obtains: 

\vspace{0.2cm}

\begin{equation}
{\bf L}_{\langle {\tilde \A}_{-1} ,{\tilde \A}_0 \rangle}(\Ee \otimes \Ll _{-\omega _1}) = \G ,
\end{equation}

\vspace{0.2cm}

where $\G$ is obtained as a unique non--trivial extension

\vspace{0.2cm}

\begin{equation}
0\rightarrow \Ll _{-3\omega _1-\omega _4}\rightarrow \G \rightarrow \Psi _1^{\omega _1}\otimes \Ll _{-\omega _1-2\omega _4} \rightarrow 0.
\end{equation}

\vspace{0.2cm}

This sequence immediately gives that the left mutation ${\bf L}_{\langle \Ll _{-3\omega _1-\omega _3}\rangle}\G$ is isomorphic to $\Psi _1^{\omega _1}\otimes \Ll _{-\omega _1-2\omega _4}$. Further, $\Hom ^{\bullet}(\Ll _{-2\omega _1-2\omega _4},\Psi _1^{\omega _1}\otimes \Ll _{-\omega _1-2\omega _4})=\nabla _{\omega _2}$, and one obtains 

\vspace{0.2cm}

\begin{equation}
0\rightarrow \Psi _2^{\omega _1}\otimes \Ll _{-2\omega _1-2\omega _4}\rightarrow \nabla _{\omega _2}\otimes \Ll _{-2\omega _1-2\omega _4}\rightarrow \Psi _1^{\omega _1}\otimes \Ll _{-\omega _1-2\omega _4}\rightarrow 0;
\end{equation}

\vspace{0.2cm}

that is, ${\bf L}_{\langle \Ll _{-2\omega _1-2\omega _4}\rangle}(\Psi _1^{\omega _1}\otimes \Ll _{-\omega _1-2\omega _4})=\Psi _2^{\omega _1}\otimes \Ll _{-2\omega _1-2\omega _4}[1]$. Finally, one calculates the group $\Hom ^{\bullet}(\Ll _{-\omega _1-3\omega _4},\Psi _2^{\omega _1}\otimes \Ll _{-2\omega _1-2\omega _4}[1])={\sf k}$, and thus the left mutation ${\bf L}_{\langle \Ll _{-\omega _1-3\omega _4}\rangle}(\Psi _2^{\omega _1}\otimes \Ll _{-2\omega _1-2\omega _4}[1])$ = ${\bf L}_{{\langle {\tilde \A}_{-2},\tilde \A}_{-1} ,{\tilde \A}_0 \rangle}(\Ee \otimes \Ll _{-\omega _1})
$ is given by a unique non--trivial extension

\vspace{0.2cm}

\begin{equation}
0\rightarrow \Psi _2^{\omega _1}\otimes \Ll _{-2\omega _1-2\omega _4}\rightarrow {\mathcal H} \rightarrow \Ll _{-\omega _1-3\omega _4}\rightarrow 0.
\end{equation}

\vspace{0.2cm}

Denote  ${\tilde \Hh} =\Hh \otimes \omega _{X_5}^{-1}=\Hh \otimes \Ll _{4(\omega _1+\omega _4)}$. Finally, by Lemma \ref{lem:mutations_canonical_class}, one has 

\vspace{0.2cm}

\begin{equation}
{\bf R}_{\langle \A _{0},\A _{1},\A _2, \A _3\rangle}(\Ee \otimes \Ll _{-\omega _1})={\bf L}_{\langle {\tilde \A}_{-2},{\tilde \A}_{-1},{\tilde \A}_{0}\rangle}(\Ee \otimes \Ll _{-\omega _1})\otimes \omega _{X_5}^{-1} = {\tilde \Hh}[3-7] = {\tilde \Hh}[-4] .
\end{equation}

\vspace{0.2cm}

Observe that $\Lambda ^2\Ee \otimes \Ll _{-\omega _1} = \Ee ^{\ast}\otimes \Ll _{-\omega _4}$, and that the bundles $\Ee ^{\ast}\otimes \Ll _{-\omega _4}$ and $\Ee ^{\ast}\otimes \Ll _{-\omega _4}$ are interchanged under the automorphism of the Dynkin diagram ${\bf A}_4$. Thus, the right dual bundle of $\Lambda ^2\Ee \otimes \Ll _{-\omega _1}$ is isomorphic to a unique non--trivial extension

\vspace{0.2cm}

\begin{equation}
0\rightarrow \Psi _2^{\omega _4}\otimes \Ll _{2\omega _1+2\omega _4}\rightarrow {\tilde \Kk} \rightarrow \Ll _{\omega _1+3\omega _4}\rightarrow 0.
\end{equation}

\vspace{0.2cm}

Similarly to Theorem \ref{th:F-theorem_X_4}, one obtains:

\begin{theorem}\label{th:F-theorem_X_5}
The bundle ${\sf F}_{\ast}\Oo _{X_5}$ decomposes into the direct sum of vector bundles with indecomposable summands being isomorphic to:

\vspace*{0.2cm}

\begin{eqnarray}\label{eq:Frobrightdualcoll_X_4}
&  \Oo _{X_5}, \quad \Ll _{-\omega _1}, \quad \Ll _{-\omega _4},  \quad \Ll _{-2\omega _1}, \quad \Ll _{-\omega _1-\omega _4}, \quad \Ll _{-2\omega _4},
\quad   \Ll _{-3\omega _1}, \quad \Ll _{-2\omega _1-\omega _4}, \quad 
\Ll _{-\omega _1-2\omega _4}, \quad \Ll _{-3\omega _4}, \nonumber \\
& (\Psi _3^{\omega _1})^{\ast}\otimes \Ll _{-\omega _1-3\omega _4}, \quad {\tilde \Hh}^{\ast}, \quad {\tilde \Kk}^{\ast}, \quad 
(\Psi _3^{\omega _1})^{\ast}\otimes \Ll _{-3\omega _1-\omega _4},\nonumber \\
& \Ll _{-\omega _1-3\omega _4}, \quad \Ll _{-2\omega _1-2\omega _4}, \quad
 \Ll _{-3\omega _1-\omega _4}, \quad \Ll _{-2\omega _1-3\omega _4}, 
\quad  \Ll _{-3\omega _1-2\omega _4}, \quad \Ll _{-3(\omega _1+\omega _4)} \nonumber
\end{eqnarray}

\vspace*{0.2cm}

The multiplicity spaces at each indecomposable summand are isomorphic, respectively, to  the cohomology of ${\sf F}^{\ast}$ of the corresponding terms of (\ref{eq:F-decomposition_on_X_5}).
\end{theorem}

\vspace*{0.5cm}

\section{The general case}\label{sec:adjoint_var}

\vspace{0.5cm}

Given the results of Section \ref{sec:adjoint_var_rank3}, one naturally arrives at the following conjecture:

\begin{conjecture}\label{conj:Frobenius_adjoint_type_A}
Let $n\in \mathbb N$. Consider the following collection of subcategories ${\tilde \A},\B, \A$ of $\Dd ^b(X_n)$, where 

\vspace{0.2cm}

\begin{equation}
\A = \langle \A _{k}\rangle, \quad 0 < k\leq n-1 ,
\end{equation}

\vspace{0.2cm}

where $\A _{0}=\langle \Oo _{X_n}\rangle$, and $\A _{k}$ for $k<0$ is defined inductively as the left mutation of the subcategory generated by $\Ll _{i\omega _1+(k-i)\omega _{n-1}}$ for $0\leq i\leq k$ through the subcategory generated by $\A _{l}$ for $0\leq l <k$. The subcategory $\B$ is generated by an exceptional collection 

\vspace{0.2cm}

\begin{equation}\label{eq:block_B}
\B = \langle \Lambda ^i\Ee\otimes \Ll _{-\omega _1}\rangle, \quad 0\leq i\leq n-2.
\end{equation}

\vspace{0.2cm}

Finally, the subcategory ${\tilde \A}$ is defined to be 
\vspace{0.2cm}

\begin{equation}
{\tilde \A} = \langle {\tilde \A} _k\rangle, \quad -n +2 < k\leq 0,
\end{equation}

\vspace{0.2cm}

where ${\tilde \A} _0=\langle \Ll _{-\omega _1-\omega _n}\rangle$, and ${\tilde \A} _k$
for $k<0$ is defined inductively as the right mutation of the subcategory generated by $\Ll _{(-k+i)\omega _1-i\omega _{n-1}}\otimes \Ll _{-\omega _1-\omega _n}$ for $0\leq i\leq -k$ through the subcategory generated by ${\tilde \A}_l$ for $k< l\leq 0$.

Then the collection of subcategories $\langle {\tilde \A},\B, \A\rangle$ is a semiorthogonal decomposition of $\Dd ^b(X_n)$ that satisfies the conditions of Theorem \ref{th:Frobdecomposclaim}. The set of indecomposable summands of ${\sf F}_{\ast}\Oo _{X_n}$ consists of the terms of the right dual decomposition to 
$\langle {\tilde \A},\B, \A\rangle$.

\end{conjecture}

Given the above considerations, proving the conjecture essentially reduces to the following statement: ${\rm H}^k(X_n,{\sf F}^{\ast}(\Lambda ^i\Ee\otimes \Ll _{-\omega _1}))=0$ for $k\neq n-1$ and $0\leq i\leq n-2$ (that is, the cohomology of Frobenius pull--backs of the bundles from the block $\B$ in (\ref{eq:block_B}) can be non--trivial only in the prescribed degree).

\end{document}